\documentclass{article}

\usepackage{amsfonts}
\usepackage{amsthm}
\usepackage{yhmath}
\usepackage{amssymb}
\usepackage{faktor}
\usepackage{authblk}

\usepackage[letterpaper,top=2cm,bottom=2cm,left=3cm,right=3cm,marginparwidth=1.75cm]{geometry}

\usepackage{amsmath}
\usepackage{graphicx}
\usepackage[colorlinks=true, allcolors=blue]{hyperref}
\newtheorem{theorem}{Theorem}
\newtheorem{lemma}{Lemma}
\newtheorem{corollary}{Corollary}

\newtheorem{definition}{Definition}

\title{Growth function for an $n$-valued dynamics}
\author{M. Chirkov \thanks{mikhlchirkov@gmail.com}}
\affil{Centre of Integrable Systems, P.G. Demidov Yaroslavl State University, Yaroslavl, Russia}
\date{}

\begin{document}
\maketitle
\begin{abstract}
This article answers the question of V.M. Buchstaber about the growth function of a particular $n$-valued group. This question is closely related to discrete integrable systems. In this paper, we will find a formula for the growth function in the case when $n$ is prime. In addition, we will prove a polynomial asymptotic estimate for the growth function in the general case. Finally, we will pose new problems and hypotheses about growth functions.
\end{abstract}
\section{Introduction}
The theory of $n$-valued groups was introduced by S. P. Novikov and V. M. Buchstaber in 1971 \cite{Buhtaber1971} in connection with the study of the theory of characteristic classes of vector bundles, in which the product of two elements is an $n$-multiset with multiplicities taken into account. Later, the theory of 2-valued formal groups was developed, which found applications to addition theorems related to elliptic functions \cite{Bukhshtaber1990}. Further, V. M. Buchstaber developed the topological and algebraic theory of $n$-valued groups, thanks to which it was possible to find a solution \cite{Buchstaber2004} of the problem of rings of functions on a space of symmetric power.

We introduce an $n$-valued multiplication \cite{Buchstaber2006} over the field $\mathbb{C}$ by the following rule:
\begin{equation}
    x \ast y = [(\sqrt[n]{x} + \varepsilon^r \sqrt[n]{y})^n, r=1,\ldots,n], \text{ where } \varepsilon = e^{\frac{2 \pi i}{n}}.
\end{equation}
Consider the function $T: \mathbb{C} \to (\mathbb{C})^n$, which defines the $n$-valued dynamics \cite{Veselov1991, Buchstaber1996} as follows:
\begin{equation}\label{dynamics}
    T(y) = x \ast y \text{ for some fixed } x \in \mathbb{C}.
\end{equation}
Then, in a natural way, we can define $T^k(y)$ as the result of applying the function $T$ to all elements of $T^{k-1}(y)$. For example, $T^2(y) = x \ast (x \ast y)$.
The question that was first posed by V. M. Buchstaber in his work \cite{Buchstaber2006} is how many different elements will be in the set $\{ 0 \} \, \cup \, T(0) \cup \ldots \cup \, T ^k(0)$. Let us denote this number $P_n(k)$. This function is related to counting the number of elements of the cyclic $n$-valued group $\langle x \rangle$ generated by an arbitrary $x \in \mathbb{C}$ after $k$ iterations. The dynamics of a cyclic $n$-valued group was first studied in \cite{Veselov1992}. V.M. Buchstaber and A.P. Veselov found \cite{Buchstaber2019} an application of the theory of 2-valued groups and their dynamics to the classical problem of adding points on elliptic curves.

In the case of a free group without bonds, growth with such dynamics turns out to be exponential. The fact of polynomial growth is not accidental and is related to the integrability of the corresponding systems with discrete time. An integrable system closely related to the 2-valued group is an integrable billiard \cite{Buchstaber2018}.

In the next section, we will introduce all the necessary definitions for describing the dynamics of \eqref{dynamics}. In Section 3, we prove an explicit formula for the growth function in the case where $n$ is a prime number. Then, in section 4, some results of numerical calculations of growth polynomials will be described. Finally, in section 5 we will make some new hypotheses and pose some questions.

\section{Description of the dynamics}
Note that $T(0) = [x, \ldots, x]$, i.e. on the first iteration the new element is only one, namely $x$. We assume that $P_n(0) = 1$, since the element $0$ is put there by definition, and $P_n(1) = 2$, since we always get only $x$ on the first iteration. Let us now look at the second iteration:
\begin{equation*}
    (\sqrt[n]{x} + \varepsilon^r \sqrt[n]{x})^n = x (1 + \varepsilon^r)^n.
\end{equation*}
That is, at the second iteration, only elements of the form $\alpha x$ appear, where $\alpha \in \mathbb{C}$. At the next iteration, we have
\begin{equation*}
    (\sqrt[n]{x} + \varepsilon^r \sqrt[n]{\alpha x})^n = x (1 + \varepsilon^r \sqrt[n]{k})^n.
\end{equation*}
Thus, we again obtain elements of the form $\overline{\alpha} x$. So $x$ can always be taken out of the bracket and only the following dynamics can be considered
\begin{equation*}
    \alpha_{r_1,\ldots,r_k} = (1 + \varepsilon^{r_k} \sqrt[n]{\alpha_{r_1, \ldots, r_{k-1}}})^n,
\end{equation*}
since the new number in dynamics is given by the number from the previous iteration and the choice of the parameter $r$. If we write down all the previous members of the sequence, we obtain:
\begin{equation}
    \alpha_{r_1,\ldots,r_k} = (1 + \varepsilon^{r_k} + \varepsilon^{r_k + r_{k-1}} + \ldots + \varepsilon^{r_k + r_{k-1} + \ldots + r_1})^n.
\end{equation}
Further, it will be convenient for us to change the parametrization as follows.
\begin{lemma}\label{lemma1}
    For any $\{a_i\} \in \mathbb{Z}^{k+1}_n$ there exists $\{r_i\} \in \mathbb{Z}^k_n$, such that
    $$
    \alpha_{r_1,\ldots,r_k} = (\varepsilon^{a_1} + \varepsilon^{a_2} + \ldots + \varepsilon^{a_{k+1}})^n
    $$
\end{lemma}
\begin{proof}
    We use the fact that $\varepsilon^n = 1$, and multiply $(\varepsilon^{a_1} + \varepsilon^{a_2} + \ldots + \varepsilon^{a_{k+1}})^n$ by $\varepsilon^{-a_1 n}$:
    $$
    (\varepsilon^{a_1} + \varepsilon^{a_2} + \ldots + \varepsilon^{a_{k+1}})^n = (1 + \varepsilon^{a_2 - a_1} + \varepsilon^{a_3 - a_1} + \ldots + \varepsilon^{a_{k+1} - a_1})^n.
    $$
    Thus, we obtain the following system of equalities modulo $n$:
    $$
    \begin{cases}
    r_k \equiv a_2 - a_1, \\
    r_k + r_{k-1} \equiv  a_3 - a_1, \\
    \ldots \\
    r_k + \ldots + r_1 \equiv a_{k+1} - a_1
    \end{cases}.
    $$
    It possesses the folowing solution:
    $$
    \begin{cases}
    r_k = a_2 - a_1 \mod n, \\
    r_{k-1} = a_3 - a_2 \mod n, \\
    r_{k-2} = a_4 - a_3 \mod n, \\
    \ldots \\
    r_1 = a_{k+1} - a_k \mod n
    \end{cases}
    $$
\end{proof}
Now define $A_j \stackrel{def}{=} \{ (\varepsilon^{a_1} + \ldots + \varepsilon^{a_j})^n \mid a_i \in \mathbb{Z}_n \}$, $ A_0 \stackrel{def}{=} \{ 0 \}$. Then, by Lemma $\ref{lemma1}$, we have $P_n(k) = | A_0 \cup \ldots \cup A_k | $. Let us split these sets into disjoint ones: $\widetilde{A_j} \stackrel{def}{=} A_j \setminus {A_0 \cup \ldots \cup A_{j-1}}$, thus $P_n(k) = |\widetilde{A_0}| + \ldots + |\widetilde{A_k}|$. It will be convenient to introduce function $Q_n(k) \stackrel{def}{=} P_n(k) - P_n(k-1) = |\widetilde{A_k}|$. It shows how many elements are added at the new step of dynamics.

Next, we focus on the description of the elements of the set $\widetilde{A_j}$ in the case where $n$ is a prime number, which will be denoted by $p$ in what follows.

\section{Main result for prime $p$}

The elements of the set $A_k$ can be specified a bit differently using  grouping by powers of $\varepsilon$:
$$
A_k = \{ (b_0 + b_1 \varepsilon + \ldots + b_{p-1} \varepsilon^{p-1})^p \mid b_0 + \ldots + b_{p-1} = k, \ b_i \in \mathbb{Z}_{\geq 0}\}
$$
The following classical definitions follow naturally from this notation.
\begin{definition}
An ordered set of numbers $\lambda = (b_0, \ldots, b_{n-1})$ is called a composition of $m$ into $n$ parts, if $b_i \in \mathbb{Z}_{>0}$ and $b_0 + \ldots + b_{n-1} = m$.
\end{definition}
\begin{definition}
    $C(m, n)$ is the set of compositions of $m$ on $n$ positive terms.
\end{definition}
Number compositions are a classical constructions. It is known that $|C(m, n)| = {{m-1} \choose {n-1}}$. It can be seen from our construction that we need to allow $b_i = 0$, so the following definitions are required.

\begin{definition}
An ordered set of numbers $\lambda = (b_0, \ldots, b_{n-1})$ is called a weak composition of $m$ into $n$ parts, if $b_i \in \mathbb{Z}_{\geq 0} $ and $b_0 + \ldots + b_{n-1} = m$.
\end{definition}
\begin{definition}
   $\overline{C(m, n)}$ is the set of weak compositions of $m$ on $n$ non-negative terms.
\end{definition}
\begin{lemma}\label{lemma2}
    $C(m+n, n) \cong \overline{C(m, n)}$
\end{lemma}
\begin{proof}
Any composition on the left side can be associated with a weak composition on the right, if 1 is subtracted from each term. Similarly, to the right side, any weak composition on the right side can be compared to a composition on the left if one is added to each term.
\end{proof}

Hence, we obtain that $| \overline{C(m, n)} | = {{m + n - 1} \choose {n-1}}$. Next, we introduce the following equivalence relation. We say that $z_1 \sim z_2$, if $z_1 = \varepsilon z_2$. We need this equivalence relation to get rid of exponentiation to the power $p$ of elements of the set $A_k$, because $z_1 \sim z_2 \iff z_1^p = z_2^p$. So we can define the following set 
$$B_k \stackrel{def}{=} \{ b_0 + b_1 \varepsilon + \ldots + b_{p-1} \varepsilon^{p-1} \mid b_0 + \ldots + b_{p-1} = k, \ b_i \in \mathbb{Z}_{\geq 0}\}.$$
In connection with the introduced equivalence relation, we have that $A_k = B_k/ z \sim \varepsilon z$. Now we want to describe the elements of the set $\widetilde{B_k}$, which is defined analogously to $\widetilde{A_k}$ by splitting into a disjunctive union.
In addition, we need the concept of cyclic composition, because the multiplication of an element of a set $B_k$ by $\varepsilon$ corresponds to a cyclic shift $(b_{p-1}, b_0, \ldots, b_{p-2})$.
\begin{definition}
    $
\left\langle
    \begin{matrix}
      m  \\
      n  \\
    \end{matrix}
\right\rangle
$ is the number of equivalence classes $C(m, n)$ relative to cyclic shift.
\end{definition}
To calculate the number of equivalence classes, the formula is known \cite{KnopfmacherRobbins}:
\begin{equation}
\left\langle
    \begin{matrix}
      m  \\
      n  \\
    \end{matrix}
\right\rangle
= \dfrac{1}{m} \sum\limits_{j | \gcd( m, n) } \varphi(j) {m/j \choose n/j}.
\end{equation}
\begin{lemma}\label{lemma3}
    Each element of the set $\widetilde{A_k}$ is uniquely represented by the weak composition equivalence class $(b_0, \ldots, b_{p-1})$ relative to cyclic shift, and in this representation there is an index $j$, that $b_j = 0$.
\end{lemma}
\begin{proof}
    To begin with, suppose that, on the contrary, there is a number $z_0 = \tilde{b}_0 + \ldots + \tilde{b}_{p-1} \varepsilon^{p-1} \in \widetilde{B_k}$ (the same as $z^n_0 \in \widetilde{A_k}$), that $\min b_i = d > 0$. For any $p$ (not only prime number) there is always a relation on the roots of unity $1 + \varepsilon + \ldots + \varepsilon^{p-1} = 0$, then $z_0 = (\tilde{b}_0 - d) + \ldots + (\tilde{b}_{p-1} - d) \varepsilon^{p-1}$. Hence we get that $z_0 \in B_{k - p d}$, which is contrary to $z_0 \in \widetilde{B_k}$, because by construction we have $\widetilde{B_k} \cap B_{k - p d} = \varnothing$.

    Let us now show that the representation is unique up to a cyclic shift of $(b_0, \ldots, b_{p-1})$. Assume, on the contrary, that $z_0 = \hat{b}_0 + \ldots + \hat{b}_{p-1} \varepsilon^{p-1} = \hat{d}_0 + \ldots + \hat{d}_{p-1} \varepsilon^{p-1}$, where $(\hat{b}_0, \ldots, \hat{b}_{p-1})$ and $(\hat{d}_0, \ldots, \hat{d}_{p-1})$ are different non-equivalent weak compositions of $k$ on $p$ parts, and without loss of generality we can assume that $\hat{b}_{0} = 0$ and $\hat{d}_{0} = 0$. Subtracting one representation from the other, we obtain that
    $$
    0 = (\hat{b}_1 - \hat{d}_1) \varepsilon + \ldots + (\hat{b}_{p-1} - \hat{d}_{p-1}) \varepsilon^{p-1}.
    $$
    From the theory of cyclotomic fields \cite{Ireland1982} we have, that the degree of field extension $\mathbb{Q}(\varepsilon)$, where $\varepsilon = e^{2 \pi i / n}$ is equal to $\varphi(n)$, where $\varphi$ is the Euler totient function. Since we are working with a prime number $p$, we have $\varphi(p) = p-1$. It means that $p-1$ successive roots of unity $\varepsilon, \varepsilon^2, \ldots, \varepsilon^{p-1}$ are linearly independent over the field $\mathbb{Q}$. Hence we get a contradiction, since $\hat{b}_i = \hat{d}_i, i = 0, \ldots, p-1$.
\end{proof}
In connection with the lemma just proved, it is useful to introduce the following definition.
\begin{definition}
    We denote $\widetilde{C(k, p)}$ as the set of equivalence classes with respect to cyclic shifts of weak compositions $\overline{C(k, p)}$, that contain at least one zero.
\end{definition}
From lemma \ref{lemma3} we obtain that $\widetilde{A_k} \cong \widetilde{C(k, p)}$.
\begin{lemma}
$\widetilde{C(k, p)} \ \sqcup \ C(k, p)/\!\sim \ \cong C(k+p, p)/\!\sim$, where $\sim$ is an equivalence relation with respect to cyclic shift.
\end{lemma}
\begin{proof}
    The disjoint union symbol does indeed take place here, because ordinary compositions contain only positive integers by definition, but representatives of equivalence classes from $\widetilde{C(k, p)}$ have at least one zero. In total, the left side contains the equivalence classes of arbitrary weak compositions of $k$ on $p$ parts. Further on the lemma $\ref{lemma2}$ we establish the final isomorphism.
\end{proof}
Hence, we obtain that $$Q_k(p) = | \widetilde{A_k} | = | \widetilde{C(k, p)} | = | C(k+p, p)/\!\sim| \ - \ | C(k, p)/\!\sim | = \left\langle
    \begin{matrix}
      k+p  \\
      p  \\
    \end{matrix}
\right\rangle - \left\langle
    \begin{matrix}
      k  \\
      p  \\
    \end{matrix}
\right\rangle.$$ 
When $k < p$, it is natural to set $\left\langle
    \begin{matrix}
      k  \\
      p  \\
    \end{matrix}
\right\rangle = {k \choose p} = 0$. Thus, we have proved the following theorem.

\begin{theorem}
    If $p$ is a prime number, then $Q_p(k) = \left\langle
    \begin{matrix}
      k+p  \\
      p  \\
    \end{matrix}
\right\rangle - \left\langle
    \begin{matrix}
      k  \\
      p  \\
    \end{matrix}
\right\rangle$.
\end{theorem}
\begin{corollary}
    If $p$ is a prime number, then $$Q_p(k) = \dfrac{1}{p} \left[ {{k + p - 1} \choose {p-1}} - {{k-1} \choose {p-1}} \right].$$
\end{corollary}
\begin{proof}
    Since $p$ is prime $\gcd(p, k)$ is equal to 1 or $p$:
    $$
    \left\langle
    \begin{matrix}
      k  \\
      p  \\
    \end{matrix}
\right\rangle = 
\begin{cases}
    \dfrac{1}{k} {k \choose p}, \ p \nmid k, \\ \\
    \dfrac{1}{k} \left( {k \choose p} + \varphi(p) \dfrac{k}{p} \right), \ p \mid k
\end{cases}
    \left\langle
    \begin{matrix}
      k+p  \\
      p  \\
    \end{matrix}
\right\rangle = 
\begin{cases}
    \dfrac{1}{k+p} {{k+p} \choose p}, \ p \nmid k, \\ \\
    \dfrac{1}{k+p} \left( {{k+p} \choose p} + \varphi(p) \dfrac{k + p}{p} \right), \ p \mid k
\end{cases}
    $$
In both cases the same result is obtained:
$$
Q_p(k) = \dfrac{1}{k+p} {{k+p} \choose p} - \dfrac{1}{k} {k \choose p}.
$$
We can expand the binomial coefficients by the definition and get the formula from the statement:
$$
Q_p(k) = \dfrac{1}{k+p} \dfrac{(k+p)!}{k! \ p!} - \dfrac{1}{k} \dfrac{k!}{(k-p)! \ p!} = \dfrac{(k+p-1)!}{k! \ p!} - \dfrac{(k-1)!}{(k-p)! \ p!} = \dfrac{1}{p} \left[ {{k + p - 1} \choose {p-1}} - {{k-1} \choose {p-1}} \right].
$$
\end{proof}
\begin{corollary}
If $p$ is a prime number, then $$P_p(k) = 1 + \sum\limits_{j = 1}^{k} \left\langle
    \begin{matrix}
      j+p  \\
      p  \\
    \end{matrix}
\right\rangle - \left\langle
    \begin{matrix}
      j  \\
      p  \\
    \end{matrix}
\right\rangle = 1 + \dfrac{1}{p} \sum\limits_{j = 1}^{k} {{j+p-1} \choose {p-1}} - {{j-1} \choose {p-1}}.$$
\end{corollary}
\begin{proof}
    It is because $P_p(k) = P_p(0) + Q_p(1) + \ldots + Q_p(k)$.
\end{proof}
\begin{corollary}
    If $p$ is a prime number, then $Q_p(k) = O(k^{p-2})$ and $P_p(k) = O(k^{p-1})$ for $k \to +\infty$.
\end{corollary}
\begin{proof}
    This result follows from the formula
    $$
    Q_p(k) = \dfrac{1}{p} \left[ {{k + p - 1} \choose {p-1}} - {{k-1} \choose {p-1}} \right].
    $$
    To do this, we expand the binomial coefficients
    $$
    \dfrac{1}{p}{{k + p - 1} \choose {p-1}} - {{k-1} \choose {p-1}} = \dfrac{1}{p!} \left[ (k + p - 1) \ldots (k + 1) - (k-1) \ldots (k - p + 1) \right].
    $$
    From this we obtain that the coefficient at $k^{p-1}$ is equal to 0, hence $Q_p(k) = O(k^{p-2})$.
    The asymptotic for $P_p(k)$ can be obtained from a rough estimate $P_p(k) \leqslant 1 + k Q_p(k)$.
\end{proof}

\section{Numerical calculation of growth polynomials}
\begin{lemma}
    $Q_n(k) = O(k^{n-2})$ and $P_n(k) = O(k^{n-1})$ for $k \to +\infty$.
\end{lemma}
\begin{proof}
    In the case of an arbitrary $n$, Lemma $\ref{lemma3}$ is only partially true; that is, the part about uniqueness is not true. So there is an injection from $\widetilde{A_k}$ into $\widetilde{C(k, n)}$, hence $$Q_n(k) \leqslant \left\langle
    \begin{matrix}
      k+n  \\
      n  \\
    \end{matrix}
\right\rangle - \left\langle
    \begin{matrix}
      k  \\
      n  \\
    \end{matrix}
\right\rangle$$
If $\gcd(k, n) = 1$ then, as we have already found out in the proof, in the case of a prime $n$ we obtain $Q_n(k) = O(k^{n-2})$ for $k \to + \infty$.
In the general case, additional terms of the form ${{(k+n)/j} \choose {n/j}}$ and ${{k/j} \choose {n/j}}$ are added, but they have smaller asymptotic of $O(k^{n/j})$.
\end{proof}
Now, we calculate the interpolation polynomial $\overline{P_n(k)}$ for $P_n(k)$ over the nodes $0, \ldots, n-1$. Since from the estimate we have $P_n(k) = O(k^{n-1})$, for prime $n$ we can guarantee that $\overline{P_n(k)} \equiv P_n(k)$:
\begin{align*}
    \overline{P_2(k)} &= k+1, \\
    \overline{P_3(k)} &= 1 + \frac{1}{2}(k + k^2) = \overline{P_4(k)} = \overline{P_6(k)}, \\
    \overline{P_5(k)} &= 1 + \frac{5}{12}(k + k^2) + \frac{1}{24}(k + k^2)^2, \\
    \overline{P_7(k)} &= 1 + \frac{7}{20}(k + k^2) + \frac{13}{180}(k + k^2)^2 + \frac{1}{720}(k + k^2)^3, \\
    \overline{P_8(k)} &= 1 + \frac{1}{3}(k + k^2) + \frac{1}{12}(k + k^2)^2, \\
    \overline{P_9(k)} &= 1 + \frac{3}{10}(k + k^2) + \frac{11}{120}(k + k^2)^2 + \frac{1}{240}(k + k^2)^3, \\
    \overline{P_{10}(k)} &= 1 + \frac{1}{4}(k + k^2) + \frac{1}{8}(k + k^2)^2 = \overline{P_{12}(k)}.
\end{align*}
\section{Conclusions and further study}
We found an interesting relation of cyclic compositions to the task of counting images of n-valued map $T^k(0)$ and using this relation found a specific formula for the case, when $n$ is prime. From our numerical results, the following hypotheses and questions arise:
\begin{enumerate}
    \item The degree of $\overline{P_n(k)}$ is $\varphi(n)$ (we see it from numerical results), we suggest that it is related to the fact that the degree of cyclotomic field $\mathbb{Q}(\varepsilon)$ as an extension of $\mathbb{Q}$ is equal to $\varphi(n)$;
    \item We have the same polynomials for some $n$, we want to know for which exactly;
    \item If $n$ is prime, the coefficient at the highest degree is $\frac{1}{(n-1)!}$;
    \item Polynomials depend only on $(k + k^2)$ (or equivalently ${{k+1} \choose 2}$) (exception -- $n=2$).
    \item Polynomials $\overline{P_n(k)}$ are integer-valued (it can be easily proven using Theorem 3.2.1 from \cite{Prasolov2004}). So we can write any integer-valued polynomial $p_d$ of degree $d$ as
    $$
    p_d(x) = c_0 {x \choose d} + c_1 {x \choose {d-1}} + \ldots + c_d, \ c_j \in \mathbb{Z}
    $$
    We are interested in $c_j$ for our polynomials $\overline{P_n(k)}$.
\end{enumerate}
In further works we are going to prove hypothesis 1 (that $P_n(k)$ is polynomial and $\deg P_n(k) = \varphi(n)$), and answer some questions about the form of $P_p(k)$, where $p$ is prime.

\section*{Acknowledgements}
Special thanks to V. M. Buchstaber for formulation of the problem, our fruitful discussions and his helpful advises.
This work was carried out within the framework of a development programme for the Regional Scientific and Educational Mathematical Center of the Yaroslavl State University with financial support from the Ministry of Science and Higher Education of the Russian Federation (Agreement on provision of subsidy from the federal budget No. 075-02-2023-948).

\bibliographystyle{plain}
\bibliography{main}
\end{document}